\documentclass{article}

\usepackage{subfigure}
\usepackage{tikz}
\usetikzlibrary{graphs,graphs.standard,calc}
\usepackage[utf8]{inputenc}

\usepackage[utf8]{inputenc}
\usepackage[english]{babel}
\usepackage{amsmath}
\usepackage{times}
\usepackage{graphicx}
\usepackage{float}
\usepackage[margin=1in]{geometry}
\usepackage{blindtext}
\usepackage{amssymb}
\usepackage{tikz}
\usepackage{amsthm}
\usepackage{mathtools}
\usepackage{textcmds}

\usepackage{comment}
\usepackage[mathlines]{lineno}
\newtheorem{theorem}{Theorem}[section]

\newtheorem{lemma}[theorem]{Lemma}
\newtheorem{conjecture}[theorem]{Conjecture}
\theoremstyle{remark}

\newtheorem{claim}{Claim}[theorem]
\theoremstyle{definition}

\title{The spectral even cycle problem}
\date{\today}
\author{Sebastian Cioab\u{a}\thanks{Department of Mathematical Sciences, University of Delaware, \texttt{cioaba@udel.edu}. Research partially supported by National Science Foundation grant CIF-1815922.} \and Dheer Noal Desai\thanks{Department of Mathematical Sciences, University of Delaware, \texttt{dheernsd@udel.edu}} \and Michael Tait\thanks{Department of Mathematics \& Statistics, Villanova University, \texttt{michael.tait@villanova.edu}. Research partially supported by National Science Foundation grant DMS-2011553.}}

\begin{document}

\maketitle
%\linenumbers
\begin{abstract}
    In this paper, we study the maximum adjacency spectral radii of graphs of large order that do not contain an even cycle of given length. For $n>k$, let $S_{n,k}$ be the join of a clique on $k$ vertices with an independent set of $n-k$ vertices and denote by $S_{n,k}^+$ the graph obtained from $S_{n,k}$ by adding one edge. In 2010, Nikiforov conjectured that for $n$ large enough, the $C_{2k+2}$-free graph of maximum spectral radius is $S_{n,k}^+$ and that the $\{C_{2k+1},C_{2k+2}\}$-free graph of maximum spectral radius is $S_{n,k}$. We solve this two-part conjecture.
\end{abstract}
\section{Introduction}

The Tur\' an number of a graph $F$ is the maximum number of edges in a graph on $n$ vertices with no subgraph isomorphic to $F$. We use $\mathrm{ex}(n, F)$ to denote the Tur\' an number of $F$ and $\mathrm{EX}(n, F)$ for the set of $F$-free graphs on $n$ vertices with $\mathrm{ex}(n, F)$ many edges. In 1941, Tur\' an \cite{Turan41} determined $\mathrm{ex}(n,K_{r+1})$, where $K_{r+1}$ is the complete graph on $r+1$ vertices, showing that the unique extremal graph is the complete $r$-partite graph with part sizes as balanced as possible (called the {\em Tur\'an graph} and denoted by $T_r(n)$). %Let $T_r(n)$ denote the complete $r$-partite graph on $n$ vertices, with all parts of size as equal as possible, also known as the Tur\' an graph. Tur\' an showed that $\mathrm{ex}(n,K_{r+1}) \leq e(T_r(n))$. This extended a previous result of Mantel \cite{Man07} on $\mathrm{ex}(n, K_3)$.   
Another celebrated theorem in extremal combinatorics is the Erd\H{o}s-Stone-Simonovits theorem \cite{ES66, ES46} which extends Tur\' an's theorem to all $r+1$ chromatic graphs, where $r \geq 2$. The statement is as follows,

\begin{equation} \label{eqESS}
  \mathrm{ex}(n,F) = \left( 1- \frac{1}{\chi (F) -1} + o(1) \right) 
  \frac{n^2}{2}, 
  \end{equation}
where $\chi(F)$ denotes the chromatic number of the forbidden graph $F$ and the term $o(1)$ goes to zero as $n$ goes to infinity. The Erd\H{o}s-Stone-Simonovits theorem gives the  exact asymptotics of the Tur\' an numbers for any forbidden graph $F$ with chromatic number $\chi(F) > 2$ and essentially says that the extremal graphs cannot do much better than the Tur\' an graphs. However, when the forbidden graphs are bipartite, that is $\chi(F) = 2$, we only get that $\mathrm{ex}(n,F) = o(n^2)$. Very little is known for even the simplest examples of bipartite graphs. Determining the asymptotics of $\mathrm{ex}(n,C_{2k})$ is one of the most famous Tur\' an-type open problem and is notoriously difficult. The order of magnitude is only known for $k\in \{2,3,5\}$ \cite{degeneratesurvey}, and determining it for other $k$ is called the even cycle problem.

In this paper, we study the spectral version of the Tur\' an problem for even cycles. Analogous to Tur\' an numbers, among all graphs on $n$ vertices, that do not contain $F$ as a subgraph, let $\mathrm{spex}(n, F)$ denote the maximum value of the spectral radius of their adjacency matrices. Also, let $\mathrm{SPEX}(n, F)$ denote the set of graphs with an adjacency matrix having spectral radius equal to $\mathrm{spex}(n, F)$. 
%Brualdi and Solheid  were the first to systematically introduce the problem of determining the maximum spectral radius among all adjacency matrices of graphs in some family of graphs, although several sporadic results appeared earlier.
Nikiforov \cite{Nikiforovpaths} was the first to systematically investigate spectral Tur\' an-type problems, although several sporadic results appeared earlier. In particular, Nikiforov \cite{Nikiforov07} proved that $\mathrm{SPEX}(n, K_{r+1}) = \{T_r(n)\}$. Since the average degree of a graph lower bounds the spectral radius of its adjacency matrices, Nikiforov's spectral result strengthens Tur\' an's theorem and implies that $\mathrm{EX}(n,K_{r+1}) = \{T_r(n)\}$. Further, Nikiforov \cite{NZ}, Babai and Guiduli \cite{BG}, independently obtained spectral analogues of the K\H{o}vari-S\' os-Tur\' an theorem \cite{KST54}, when forbidding a complete bipartite graph $K_{s,t}$.  Moreover, using the average degree bound for the spectral radius gives bounds that match the best improvements to the K\H{o}vari-S\' os-Tur\' an theorem, obtained by F\"uredi \cite{F2}.      

Recently, determining $\mathrm{spex}(n, F)$ for various graphs $F$ has become very popular (see \cite{cioabua2022spectral, FG, SpectralIntersectingCliques, Yongtao21, LLT, NO, SSbook, Wilf86, YWZ, ZW, ZWF}). This fits into a broader framework of {\em Brualdi-Solheid problems} \cite{BS} which investigate the maximum spectral radius over all graphs belonging to a specified family of graphs. Numerous results are known in this area (see \cite{BZ, BLL, EZ, FN, nosal1970eigenvalues, S, SAH}). In \cite{Nikiforovpaths}, Nikiforov conjectured the solution to the spectral Tur\' an problem for even cycles.
Let $S_{n,k} := K_k \vee \overline{K}_{n-k}$ and $S_{n,k}^+ := K_k \vee (\overline{K}_{n-k-2} \cup K_2)$. 
The graph $S_{n,k}^+$ has $n$ vertices and does not contain any $C_{2k+2}$, while, $S_{n,k}$ has $n$ vertices and contains neither any $C_{2k+1}$ nor $C_{2k+2}$.
Nikiforov \cite[Conjecture 15]{Nikiforovpaths} made the following two-part conjecture. 
\begin{conjecture}
\label{conjecture Nikiforov cycles}
Let $k \geq 2$ and $G$ be a graph of sufficiently large $n$.  \\
(a) if $\lambda(G) \geq \lambda(S_{n,k})$ then $G$ contains $C_{2k+1}$ or $C_{2k+2}$ unless $G = S_{n,k}$; \\
(b) if $\lambda(G) \geq \lambda(S_{n,k}^+)$ then $G$ contains $C_{2k+2}$ unless $G = S_{n,k}^+$.
\end{conjecture}

In this paper, we fully resolve both parts of Conjecture \ref{conjecture Nikiforov cycles} in the affirmative.

\begin{theorem}
\label{thm spectral turan even cycles}
Let $k \geq 2$ and $n$ be sufficiently large, then $\mathrm{SPEX}(n, C_{2k+2}) = \{S_{n, k}^+\}$.
\end{theorem}
Theorem \ref{thm spectral turan even cycles} was proved to be true for $k=2$ by Zhai and Lin \cite{zhailinforbiddenhexagon}. We settle the remaining cases with $k>2$. We also prove the following theorem which resolves Part(a) of the conjecture.
\begin{theorem}
\label{thm spectral turan subsequent cycles}
Let $k \geq 2$ and $n$ be sufficiently large, then $\mathrm{SPEX}(n, \{C_{2k+1}, C_{2k+2}\}) = \{S_{n, k}\}$.
\end{theorem}

\section{Organization and Notation}

For some fixed $k\geq 2$, let $H_k \in \mathrm{SPEX}(n, C_{2k+2})$ be  a spectral extremal graph when forbidding the even cycle of size $2k+2$ and let $H_k'\in \mathrm{SPEX}(n, \{C_{2k+1}, C_{2k+2})$ be a spectral extremal graph when forbidding both $C_{2k+1}$ and $C_{2k+2}$. All of our arguments in Sections \ref{section even cycles structural lemmas} and \ref{section even cycle structural results for extremal graphs} apply with identical proofs to both $H_k$ and $H_k'$ as we will only be using the fact that the graph is extremal and $C_{2k+2}$-free, and so for brevity we will only state the results for $H_k$ until the proofs of Theorems \ref{thm spectral turan even cycles} and \ref{thm spectral turan subsequent cycles} (Section \ref{section proofs for even cycle chapter}).

For any vertex $u$ and non-negative integer $i$, let $N_i(u)$ denote the set of vertices at distance $i$ from $u$, with $d_i(u):=|N_i(u)|$. We use $d(u) = d_1(u)$ to denote the degree of $u$. For two disjoint subsets $X, Y\subset V(H_k)$, denote by $H_k[X,Y]$ the bipartite subgraph of $H_k$ with vertex set $X\cup Y$ that consists of all the edges with one endpoint in $X$ and the other endpoint in $Y$. Let $\mathcal{E}(X,Y)$ be the edge-set of $H_k[X,Y]$ and set $e(X,Y) := |\mathcal{E}(X,Y)|$. Denote by $\mathcal{E}(X)$ the set of edges with both endpoints in $X$ and set $e(X):=|\mathcal{E}(X)|$.

For a graph $G=(V,\mathcal{E})$, we denote by $A(G)$ its adjacency matrix and by $\lambda(G)$ the spectral radius of $A(G)$. Associated to the spectral radius of the adjacency matrix of any connected graph is a unique (up to scalar multiples) entrywise positive eigenvector, commonly referred to as the Perron vector of the matrix. Since adding an edge between two disconnected components of a graph does not create any cycles and increases the spectral radius, the graphs $H_k$ and $H_k'$ must be connected. 
Let $\mathbf{\mathrm{v}}$ be the Perron vector of the adjacency matrix of $H_k$ with maximum entry $\mathbf{\mathrm{v}}_x = 1$, where $\mathbf{\mathrm{v}}_u$ denotes the coordinate of $\mathbf{\mathrm{v}}$ with respect to some vertex $u$.

We will fix a small constant $\alpha$ below, and we define $L$ to be the following set of vertices of large weight, and denote by $S$ its complement:
\begin{equation*}
L := \{u \in V(H_k) |  \mathbf{\mathrm{v}}_u > \alpha\} \text{ and } S := V(H_k) \setminus L = \{u \in V(H_k) |  \mathbf{\mathrm{v}}_u \leq \alpha\}.
\end{equation*}
Additionally, we will also use the following set in the proof of Lemma \ref{degrees of L}. Let 
\begin{equation*}
M := \{u \in V(H_k) |  \mathbf{\mathrm{v}}_u \geq \alpha/3\}.    
\end{equation*}
Finally, we define a subset of $L$ called $L'$ by 
\begin{equation*}
L':= \{u| \mathbf{\mathrm{v}}_u \geq \eta\},
\end{equation*}
where $\eta > \alpha$ is a constant defined below. For a vertex $u$, denote
\begin{equation*}
L_i(u):= L \cap N_i(u), S_i(u):= S \cap N_i(u), \text{ and } M_i(u) := M \cap N_i(u).
\end{equation*}
If the vertex is unambiguous from context, we will use $L_i$, $S_i,$ and $M_i$ instead. 

With foresight, we choose $\eta$, $\epsilon$, and $\alpha$ to be any positive constants satisfying
\begin{equation}\label{eta choice}
    \eta < \left\{\frac{1}{k+1}, 1 - \frac{1}{16k^3}, \frac{1}{4} - \frac{1}{16k^2}\right\}
\end{equation}
\begin{equation}\label{epsilon choice}
   \epsilon <  \min \left\{ \frac{1}{16k^3}, \frac{\eta}{2}, \frac{\eta}{32k^3+2}\right\}
\end{equation}
\begin{equation}\label{alpha choice}
\alpha <  \frac{\epsilon^2}{10k}.
\end{equation}

We note that many of the above inequalities are redundant, but we leave them so that it is easier to see exactly what inequalities we are using throughout the proofs. In Sections \ref{section even cycles structural lemmas} and \ref{section even cycle structural results for extremal graphs}, we prove lemmas showing the structural properties of $H_k$ and $H_k'$. We reiterate that every lemma applies to both $H_k$ and $H_k'$ with identical proofs and so the proofs are only written for $H_k$. In Section \ref{section proofs for even cycle chapter}, we complete the proofs of Theorems \ref{thm spectral turan even cycles} and \ref{thm spectral turan subsequent cycles}.

\section{Lemmas from spectral and extremal graph theory}\label{section even cycles structural lemmas}

In this section, we record several lemmas that we will use. Some calculations may only apply for $n$ large enough without being explicitly stated. We start with a standard result from linear algebra which serves as a tool to bound the spectral radius of non-negative matrices.

\begin{lemma}
\label{lower bound for spectral radius of nonnegative matrices}

For a non-negative symmetric matrix $M$, a non-negative non-zero vector $\mathbf{\mathrm{y}}$ and a positive constant $c$, if $M \mathbf{\mathrm{y}} \geq c \mathbf{\mathrm{y}}$ entrywise, then $\lambda(M) \geq c$. 
\end{lemma}

\begin{proof}
Assume that $M \mathbf{\mathrm{y}} \geq c \mathbf{\mathrm{y}}$ entrywise, with the same assumptions for $M, \mathbf{\mathrm{y}}$ and $c$ as in the statement of the theorem. Then $\mathbf{\mathrm{y}}^T M \mathbf{\mathrm{y}} \geq \mathbf{\mathrm{y}}^T c \mathbf{\mathrm{y}}$ and $\lambda(M) \geq \dfrac{\mathbf{\mathrm{y}}^T M \mathbf{\mathrm{y}}}{\mathbf{\mathrm{y}}^T  \mathbf{\mathrm{y}}} \geq c$.
\end{proof}

\begin{lemma}[Erd\H{o}s-Gallai \cite{gallai1959maximal}]
\label{extremal path}
Any graph on $n$ vertices with no subgraph isomorphic to a path on $\ell$ vertices has at most $\dfrac{(\ell-2)n}{2}$ edges.
\end{lemma}

We will be using Part B of the following lemma which appears in \cite[Lemma 1]{nikiforov2009degree}.
\begin{lemma}[Nikiforov \cite{nikiforov2009degree}]
\label{Many edges in a bipartition}
Suppose that $k \geq 1$ and let the vertices of a graph $G$ be partitioned into two sets $U$ and $W$.
\begin{enumerate}
    \item[(A)] If
    \begin{equation}
    \label{eqn Many edges in a bipartition A}
        2e(U) + e(U, W) > (2k-2)|U| + k|W|,
    \end{equation}
then there exists a path of order $2k$ or $2k+1$ with both ends in $U$.
\item[(B)] If
    \begin{equation}
    \label{eqn Many edges in a bipartition B}
        2e(U) + e(U, W) > (2k-1)|U| + k|W|,
    \end{equation}
then there exists a path of order $2k+1$ with both ends in $U$.
\end{enumerate}
\end{lemma}

Let $u$ be any vertex in $H_k$. Since our graph is $C_{2k+2}$-free, $N_1(u)$ may not contain a $P_{2k+1}$. Hence, by Lemma \ref{extremal path} we have 
\begin{equation}\label{N1 edge bound}
    e(N_1(u)) \leq \frac{2k-1}{2}d(u) < kn.
\end{equation}
Similarly the bipartite subgraph between $N_1(u)$ and $N_2(u)$ may not contain a $P_{2k+3}$, otherwise there is a $P_{2k+1}$ with both endpoints in $N_1(u)$ and hence a $C_{2k+2}$. Therefore, by Lemma \ref{extremal path} (forbidding $P_{2k+3}$ in the bipartite subgraph) and Lemma \ref{Many edges in a bipartition} (forbidding $P_{2k+1}$ with both endpoints in $N_1(u)$), we have 
\begin{equation}\label{N1 to N2 edge bound}
    e(N_1(u), N_2(u)) \leq \min\left\{ \frac{2k+1}{2}n, (2k-1)d(u) + k\left(n - d(u) - 1\right)\right\}.
\end{equation}
The spectral radius of $S_{n,k}$ gives a lower bound for $\lambda(H_k)$. We will modify an argument of Nikiforov (proof of \cite[Theorem 3]{Nikiforovpaths}) to obtain an upper bound for $\lambda(H_k)$.
\begin{lemma}\label{bounds on lambda}
$\sqrt{kn} \leq \frac{k-1 + \sqrt{(k-1)^2 + 4k(n-k)}}{2} \leq \lambda(H_k) \leq \sqrt{2k(n-1)}$.
\end{lemma}
\begin{proof}

Here the inner lower bound is precisely $\lambda(S_{n,k})$ and the first inequality on the left follows from a straightforward calculation. To prove the upper bound, let $u \in V(H_k)$. We use Lemma \ref{Many edges in a bipartition} over the graph $H_k[N_1(u) \cup N_2(u)]$ with $U = N_1(u)$ and $W = N_2(u)$. We know that $|N_1(u)| = d(u)$. Also, there cannot be any path on $2k+1$ vertices in $H_k[N_1(u) \cup N_2(u)]$ with both end points in $N_1(u)$. 
So \eqref{eqn Many edges in a bipartition B} implies that 
\begin{equation}
\label{eqn precursor to bounds on spectral radius using bipartition edges}
\begin{split}
2e(N_1(u)) + e(N_1(u), N_2(u)) &\leq (2k-1)d(u) + k d_2(u)\\ &\leq (2k-1)d(u) + k (n - d(u) - 1)\\ &= (k-1)d(u) + k (n - 1).    
\end{split}    
\end{equation}
The spectral radius of a non-negative matrix is at most the maximum of the row-sums of the matrix. Applying this result for $A^2(H_k)$ and its spectral radius $\lambda^2$ and using \eqref{eqn precursor to bounds on spectral radius using bipartition edges}, we obtain that
\begin{equation}
\begin{split}
\lambda^2 &\leq \max_{u \in V(H_k)} \big\{\sum_{w\in V(H_k)}\ A^2_{u,w} \big\} = \max_{u \in V(H_k)}\big\{\sum_{v \in N(u)}d(v)\big\}\\ &= \max_{u \in V(H_k)}\big\{d(u) + 2e(N_1(u)) + e(N_1(u), N_2(u))\}\\
&\leq kd(u) + k(n-1) \leq 2k(n-1).
\end{split}
\end{equation}
Thus, $\lambda \leq \sqrt{2k(n-1)}$.
\end{proof}

Next we determine an upper bound for the number of vertices in $L$. We use the same technique as in the proof of \cite[Lemma 8]{tait2017three}. For this work, we use the even-circuit theorem. Note that the best current bounds for $\mathrm{ex}(n, C_{2k})$ are given by He \cite{He} (see also Bukh and Jiang \cite{B}), but for our purposes, the dependence of the multiplicative constant on $k$ is not important. We use the following version because it makes the calculations slightly easier.

\begin{lemma}[Even Circuit Theorem \cite{Verstraeteevencircuit}]
\label{lemC2k}
For $k\geq 1$ and $n$ a natural number, 
\[\textup{ex}(n, C_{2k+2}) \leq 8kn^{(k+2)/(k+1)}. \]
\end{lemma}

\begin{lemma}\label{very loose bound on size of L}
$|L| \leq \dfrac{16k^{1/2}n^{(k+3)/(2k+2)}}{\alpha}$ and $|M| \leq \dfrac{48k^{1/2}n^{(k+3)/(2k+2)}}{\alpha}$. 
%\Theta(n^{(k+3)/(2k+2)})$
\end{lemma}
\begin{proof}
For any vertex $u \in V(H_k)$, we have the following equation relating the spectral radius and Perron vector entries,
\begin{equation}
\label{sum of Perron entries of neighbors}
\lambda \mathbf{\mathrm{v}}_u = \sum_{w \sim u}\mathbf{\mathrm{v}}_w.
\end{equation}
Because $\sqrt{kn}\leq \lambda$ and $v_w\leq 1$, we get that $$\sqrt{kn}\mathbf{\mathrm{v}}_u \leq \lambda \mathbf{\mathrm{v}}_u \leq d_u,$$ Summing up all these inequalities for $u\in L$, we obtain that
\[
\dfrac{|L|\sqrt{kn}\alpha}{2} \leq \frac{1}{2} \sum_{u\in L} \lambda \mathbf{\mathrm{v}}_u \leq \dfrac{1}{2}\sum_{u \in L} d_u \leq \dfrac{1}{2}\sum_{u \in V(H_k)}d_u \leq  \mathrm{ex}(n, C_{2k+2}) \leq 8kn^{(k+2)/(k+1)},
\]
which implies that
\begin{equation}
\label{bound on |L|}
    |L| \leq \dfrac{16k^{1/2}n^{(k+3)/(2k+2)}}{\alpha}. 
\end{equation}
The bound for $|M|$ is obtained similarly by replacing $\alpha$ by $\alpha/3$ everywhere above.
\end{proof}

We use the following result of Nikiforov \cite[Theorem 2]{nikiforov2009degree} to get a better upper bound for the size of $L$ in Lemma \ref{degrees of L}.   
\begin{lemma}[Nikiforov \cite{nikiforov2009degree}]\label{Degree powers in graphs with a forbidden even cycle}
Let $G$ be a graph with $n$ vertices and $m$ edges. If $G$ does not contain a $C_{2k+2}$ then
$$\sum_{u \in V(G)} d_G^2(u) \leq 2km + k(n-1)n.$$
\end{lemma}

Using Lemma \ref{bounds on lambda} we obtain the following lower bound for entries in the Perron vector of the extremal graphs by modifying a proof of Tait and Tobin (proof of \cite[Lemma 10]{tait2017three}).
\begin{lemma}
\label{minimum eigenvector entries}
For any vertex $u \in V(H_k)$, $\mathbf{\mathrm{v}}_u \geq \dfrac{1}{\lambda(H_k)} \geq \dfrac{1}{\sqrt{2k(n-1)}}$. 
\end{lemma}
\begin{proof}
Towards a contradiction, assume that there exists a vertex $u \in V(H_k)$, such that $\mathbf{\mathrm{v}}_u < \frac{1}{\lambda(H_k)}$. Then by \eqref{sum of Perron entries of neighbors}, $u$ cannot be adjacent to any vertex $x$ such that $\mathbf{\mathrm{v}}_x = 1$. Let $\hat{H_k}$ be the graph obtained by modifying $H_k$ by removing all the edges adjacent to $u$ and making $u$ adjacent to $x$. Then using the Rayleigh quotient, we have $\lambda(\hat{H_k}) > \lambda(H_k)$. Because adding a vertex of degree one to a graph cannot create a cycle, $\hat{H_k}$ does not contain any subgraph isomorphic to $C_{2k+2}$, contradicting that $H_k$ is extremal. 
\end{proof}

\section{Structural results for extremal graphs}
\label{section even cycle structural results for extremal graphs}

In this section, we will assume that $k\geq 2$ is fixed. 
We will be working with subgraphs in $H_k$ and due to lack of ambiguity we will drop $H_k$ from some notations now onward.
We will continue to use auxiliary constants $\alpha$, $\epsilon$, and $\eta$ and we will frequently assume that $n$ is larger than some constant depending only on $\alpha, k, \epsilon, \eta$. Every lemma in this section holds only for $n$ large enough.

\begin{lemma}
\label{degrees of L}

For any vertex $z \in L$, we have $d(z) \geq \frac{\alpha}{20k}n$. Also, $|L| \leq \frac{k+1}{(\alpha/20k)^2}$.

\end{lemma}
\begin{proof}
For some vertex $z \in L$ such that $\mathbf{\mathrm{v}}_z = c$, consider the following second degree eigenvalue-eigenvector equations relating $\lambda^2$, $\mathbf{\mathrm{v}}$, and entries in the $z$-th row of $A^2$:  
\[knc\leq \lambda^2 c= \lambda^2 \mathbf{\mathrm{v}}_z =\sum_{u\sim z}\sum_{w\sim u} \mathbf{\mathrm{v}}_w \leq d(z)c + 2e(N_1(z)) + \sum_{u\sim z} \sum_{\substack{w\sim u\\ w\in N_2(z)}} \mathbf{\mathrm{v}}_w \leq 2kd(z) + \sum_{u\sim z} \sum_{\substack{w\sim u\\ w\in N_2(z)}} \mathbf{\mathrm{v}}_w,  \]
where the last inequality is by \eqref{N1 edge bound}. Now assume to the contrary that there is a vertex $z\in L$ with $d(z) < \frac{\alpha}{20k} n$. Substituting this into the above equation and using $\alpha < c$ since $z\in L$, we have 
\[
(k-0.1)nc < \sum_{u\sim z} \sum_{\substack{w\sim u\\ w\in N_2(z)}} \mathbf{\mathrm{v}}_w.
\]
Next we show that many of the terms in the double sum come from vertices in $M_2(z)$ via the following claim.

\begin{claim}
There are at least $0.9nc$ terms $\mathbf{\mathrm{v}}_w$ with $w\in M_2$ in the sum \[\sum_{u\sim z} \sum_{\substack{w\sim u\\ w\in N_2(z)}} \mathbf{\mathrm{v}}_w.\]
\end{claim}
\begin{proof}
Assume to the contrary that there are less than $0.9nc$ terms $\mathbf{\mathrm{v}}_w$ where $w \in M_2$. As $\mathbf{\mathrm{v}}_w\leq 1$ for such $w$, 
\[
(k-0.1)nc < \sum_{u\sim z} \sum_{\substack{w\sim u\\ w\in N_2(z)}} \mathbf{\mathrm{v}}_w = \sum_{u\sim z} \sum_{\substack{w\sim u\\ w\in M_2}} \mathbf{\mathrm{v}}_w +  \sum_{u\sim z} \sum_{\substack{w\sim u\\ w\in S_2\setminus M_2}} \mathbf{\mathrm{v}}_w < 0.9nc + e(N_1, S_2\setminus M_2)\frac{\alpha}{3},
\]
and so 
\[
(k-1)nc <  e(N_1, S_2\setminus M_2)\frac{\alpha}{3}.
\]
From $\alpha < c$ and \eqref{N1 to N2 edge bound}, we have that
\[
(k-1)n < \frac{2k+1}{2} \frac{1}{3} n,
\]
a contradiction for $k\geq 2$. This proves our claim. \end{proof}

Therefore, $e(N_1(z), M_2(z)) \geq 0.9 n c > 0.9 n\alpha$. Because $H_k[N_1(z) \cup M_2(z)]$ contains no $P_{2k+3}$, by Lemma \ref{extremal path}, we deduce that $0.9 n \alpha \leq \dfrac{2k+1}{2}|N_1(z) \cup M_2(z)| < \dfrac{2k+1}{2}\left(\dfrac{n\alpha}{20k} + |M_2(z)|\right)$. Thus,
\[
|M_2(z)| > \left(.9\alpha - \frac{(2k+1)\alpha}{40k}\right)\left(\dfrac{2}{2k+1}\right)n.
\]

This contradicts the bound in Lemma \ref{very loose bound on size of L} for $n$ sufficiently large. Thus, for $n$ sufficiently large we have that $d(z) \geq \frac{\alpha}{20k}n$ for all $z\in L$. Combined with Lemma \ref{Degree powers in graphs with a forbidden even cycle} this gives us that $|L| \leq \frac{k+1}{(\alpha/20k)^2}$.\end{proof}

We now refine the lower bound on the degrees of vertices in $L'$. 

\begin{lemma}
\label{mindegrees of vertices in L}
If $z$ is a vertex of $L'$ with $\mathbf{\mathrm{v}}_z = c$, then $d(z) \geq cn - \epsilon n$.
\end{lemma}

\begin{proof}
Given $z\in L$ with $\mathbf{\mathrm{v}}_z = c$, observe that
\begin{align*}
    knc &\leq \lambda^2 c= \sum_{u\sim z}\sum_{w\sim u} \mathbf{\mathrm{v}}_w = d(z) c + \sum_{u \sim z} \sum_{\substack{w \sim u\\ w \not= z}} \mathbf{\mathrm{v}}_w\\
    &\leq d(z)c + %e(S_1, L_1\cup L_2)
    \left(\sum_{\substack{u\sim z\\ u\in S_1}} \sum_{\substack{w\sim u \\ w\in L_1\cup L_2}} \mathbf{\mathrm{v}}_w \right)+ 2e(S_1)\alpha + 2e(L_1) + e(S_1, L_1)\alpha + e(N_1, S_2)\alpha.
\end{align*}
Since $N_1$ is $P_{2k+1}$-free and the bipartite graph between $N_1$ and $N_2$ is $P_{2k+3}$-free,% and $|L|=O(1)$ by Lemma \ref{degrees of L}, we have by \eqref{N1 edge bound} and \eqref{N1 to N2 edge bound} that
\begin{align*}
2e(S_1) &\leq 2e(N_1)  \leq (2k-1)n\\
    e(L_1, S_1)& \leq e(N_1) < kn,\\
    e(N_1, S_2)& < 2kn.
\end{align*}

Using these and Lemma \ref{degrees of L}, we have 
\[
2e(S_1)\alpha + 2e(L_1) \leq 2e(N_1)\alpha + 2 \binom{|L|}{2} \leq (2k-1)n\alpha + |L|(|L|-1) < 2kn\alpha,
\]
where the last inequality holds for $n$ large enough.

Hence, we have 
\begin{equation}\label{upperbound on lambdasq 2}
    knc < d(z)c + \left(\sum_{\substack{u\sim z\\ u\in S_1}} \sum_{\substack{w\sim u \\ w\in L_1\cup L_2}} \mathbf{\mathrm{v}}_w \right)  + 5kn\alpha \leq d(z)c + e(S_1, L_1\cup L_2) + 5kn \alpha < d(z)c + e(S_1, L_1\cup L_2) + \frac{\epsilon^2 n}{2},
\end{equation}
by the choice of $\alpha$ in \eqref{alpha choice}.

If $d(z) \leq (c - \epsilon)n$ %for some $\epsilon > 0$
, then 
\begin{equation}
\label{degree bound using edges in some bipartite graph involving vertices of L 2}
(k-c + \epsilon) n c \leq (kn - d(z)) c \leq e(S_1, L_1 \cup L_2) + \frac{\epsilon^2 n}{2}.
\end{equation}
Since $z\in L'$ we have $c\geq \epsilon$. Rearranging and using $\epsilon \leq c \leq 1$, we get that
\begin{equation}
\label{degree bound using edges in some bipartite graph involving vertices of L 3}
e(S_1, L_1\cup L_2) \geq (k-1)nc + \frac{\epsilon^2 n}{2}.
\end{equation}

We will show that $H_k[S_1, L_1 \cup L_2]$ contains a $P_{2k+1}$ with both endpoints in $S_1$, thus contradicting the fact that $H_k$ is $C_{2k+2}$-free. To show this we prove the following claim.

\begin{claim}
\label{claim in many edges on (L, MUT)}
If $\delta:= \frac{\epsilon(\alpha/20k)^2}{k+1}$, then there are at least $\delta n$ vertices inside $S_1$ with degree at least $k$ in $H_k[S_1, L_1 \cup L_2]$.
\end{claim}

\begin{proof}
Assume to the contrary that at most $\delta n$ vertices in $S_1$ have degree at least $k$ in $H_k[S_1, L_1 \cup L_2]$. Then $e(S_1, L_1 \cup L_2) < (k-1)|S_1| + |L|\delta n \leq (k-1)(c-\epsilon)n +\epsilon n$, because $|S_1| \leq d(z)$ and by Lemma \ref{degrees of L}. Combining this with \eqref{degree bound using edges in some bipartite graph involving vertices of L 3} gives $ (k-1)nc- (k-2)n \epsilon > e(S_1, L_1\cup L_2) \geq (k-1)nc + \frac{\epsilon^2 n}{2}$, a contradiction.
\end{proof}

Hence, there is some subset of vertices $B \subset S_1$ such that any vertex in $B$ has degree at least $k$ in $H_k[S_1, L_1 \cup L_2]$ and $|B| = \delta n$. Since there are $\binom{|L|}{k}$ options for every vertex in $B$ to choose a set of $k$ neighbors from, we have that there is some set of $k$ vertices in $L_1 \cup L_2$ with at least $\delta n / \binom{|L|}{k}$ common neighbours in $B$. Therefore, by Lemma \ref{extremal path} and Lemma \ref{degrees of L}, for $n$ sufficiently large we have a path on $2k+1$ vertices with both end points in the common neighbourhood contained in $B$, a contradiction.
\end{proof}

Thus, for the vertex $x$ such that $\mathbf{\mathrm{v}}_x = 1$, we have $d(x) \geq n - \epsilon n$ and $N_1(x)$ contains all but at most $\epsilon n$ many vertices. Since every vertex in $L'$ has degree more than $\epsilon n$ (by the definition of $L'$ and Lemma \ref{mindegrees of vertices in L}), this also gives that $L '\setminus \{x\} \subset L_1(x) \cup L_2(x)$. The arguments in the proof of Lemma \ref{mindegrees of vertices in L} also allow us to show that all vertices of $L'$ have degrees close to $n$
%Perron entries close to $1$ 
and thus obtain $|L'| = k$.

\begin{lemma}
\label{neighborhood of x, L}
For any vertex $z\in L'$ with $\mathbf{\mathrm{v}}_z \geq 1-\epsilon$, we have $(k - 2\epsilon)n \leq e(S_1, L) \leq (k+\epsilon)n $.   
\end{lemma}

\begin{proof}
To obtain the lower bound we refine \eqref{upperbound on lambdasq 2}. Using $1-\epsilon \leq \mathbf{\mathrm{v}}_z \leq 1$ and $d(z) \leq n$, we get 
\[
kn(1-\epsilon) < d(z) + e(S_1, L_1\cup L_2) + \frac{\epsilon^2 n}{2} = e(S_1, L) + \frac{\epsilon^2 n}{2}.
\]

Thus $e(S_1, L) > (1-2\epsilon)kn$.

To obtain the upper bound, assume to the contrary that $e(S_1, L) > kn + \epsilon n$. 
%Thus, $e(S_1, L_1 \cup L_2) \geq (k-1)n + \epsilon n -o(n)$.
We will show that $H_k[S_1, L \setminus \{z\}]$ contains a $P_{2k+1}$ with both endpoints in $S_1$, thus contradicting the fact that $H_k$ is $C_{2k+2}$-free. To show this we prove the following claim.

\begin{claim}
\label{claim in many edges on (L, MUT) modified}
Let $\delta:= \frac{\epsilon(\alpha/20k)^2}{k+1}$. Then there are $\delta n$ vertices inside $S_1$ with degree at least $k$ in $H_k[S_1, L \setminus \{z\}]$.
\end{claim}

\begin{proof}
Assume to the contrary that at most $\delta n$ vertices of $S_1$ have degree at least $k$ in $H_k[S_1, L \setminus \{z\}]$. Then, $e(S_1, L \setminus \{z\}) < (k-1)|S_1| + |L| \delta n) \leq (k-1)n + |L| \delta n$, because $|S_1| \leq d(z)$. This contradicts our assumption that $e(S_1, L) \geq kn + \epsilon n$.
\end{proof}
Hence, there is some subset of vertices $B \subset S_1$ such that any vertex in $B$ has degree at least $k$ in $H_k[S_1, L \setminus \{z\}]$ and $|B| = \delta n$. Since there are only $\binom{|L|}{k}$ options for every vertex in $B$ to choose a set of $k$ neighbors from, we have that there is some set of $k$ vertices in $L \setminus \{x\}$ with at least $\delta n / \binom{|L|}{k}$ common neighbours in $B$. Therefore, by Lemma \ref{extremal path} and Lemma \ref{degrees of L}, for $n$ sufficiently large we have a path on $2k+1$ vertices with both end points in the common neighbourhood contained in $B$, a contradiction.\end{proof}

%Hence, there is some subset of vertices $B \subset S_1$ such that any vertex in $B$ has degree at least $k$ in $H_k[S_1, L \setminus \{z\}]$ and $|B| = \delta n$. Since $|L| = \Theta(1)$, we have that $\binom{|L|}{k} = \Theta(1)$. Thus there are only $\Theta(1)$ options for every vertex in $B$ to choose a set of $k$ neighbors from. Hence, there is some set of $k$ vertices in $L \setminus \{x\}$ with at least $\Theta(n)$ common neighbors in $B$. Therefore, by Lemma \ref{extremal path}, we have a path on $2k+1$ vertices with both end points in the common neighborhood contained in $B$, a contradiction.

\begin{lemma}
\label{precise size of L}
For all vertices $z\in L'$, we have 
$d(z) \geq \left(1-\frac{1}{8k^3}\right)n$ and $\mathbf{\mathrm{v}}_z \geq 1- \frac{1}{16k^3}$. Moreover, $|L'| = k$.
\end{lemma}
\begin{proof}
If we show that every vertex $z \in L'$ has Perron entry $\mathbf{\mathrm{v}}_z \geq 1- \frac{1}{16k^3}$, then it follows from Lemma \ref{mindegrees of vertices in L} and \eqref{epsilon choice} that $d(z) \geq \left(1 - \frac{1}{8k^3}\right)n$. If all vertices in $L'$ have degree at least $n-\frac{n}{8k^3}$, then $|L'| \leq k$, else there exists a $K_{k+1, k+1}$ in $H_k$, a contradiction. Also, if $|L'| \leq k-1$, then by \eqref{upperbound on lambdasq 2} and Lemma \ref{neighborhood of x, L}, we have
\[kn = kn \mathbf{\mathrm{v}}_x \leq \lambda^2 \leq e(S_1(x), L'(x)) + e(S_1(x), L(x)\setminus L'(x))\eta + \frac{\epsilon^2n}{2}  \leq (k-1)n + (k+\epsilon)n \cdot \eta + \frac{\epsilon^2 n}{2},\]
a contradiction by \eqref{eta choice} and \eqref{epsilon choice}. Hence, all we need to show is that every vertex in $L'$ has Perron entry at least $1 - \frac{1}{16k^3}$.

By way of contradiction, assume that $z\in L'$ and $\mathbf{\mathrm{v}}_z < 1 - \frac{1}{16k^3}$. Refining \eqref{upperbound on lambdasq 2} applied to the vertex $x$ we have 
\begin{align*}
kn &< e(S_1(x), L_1(x) \setminus \{z\}) + |N_1(x)\cap N_1(z)|\mathbf{\mathrm{v}}_z + \frac{\epsilon^2 n}{2}\\
&\leq (k+\epsilon)n - |N_1(x)\cap N_1(z)| + \left(1-\frac{1}{16k^3}\right)|N_1(x)\cap N_1(z)| + \frac{\epsilon^2 n}{2}\\
& = kn + \epsilon n + \frac{\epsilon^2 n}{2} - \frac{|N_1(x)\cap N_1(z)|}{16k^3},
\end{align*}
where the second inequality is by Lemma \ref{neighborhood of x, L} and the bound on $\mathbf{\mathrm{v}}_z$. Therefore, we have that
\[
\frac{|N_1(x)\cap N_1(z)|}{16k^3} < 2\epsilon n.
\]
On the other hand, since $z\in L'$ we have $\mathbf{\mathrm{v}}_z \geq \eta$ and so by Lemma \ref{mindegrees of vertices in L} we have $|N_1(x)\cap N_1(z)| \geq (\eta - 2\epsilon)n$. Combining the two inequalities is a contradiction by \eqref{epsilon choice}.
\end{proof}

Now that we have $|L'| = k$ and every vertex in $L'$ has degree at least $\left(1-\frac{1}{8k^3}\right)n$, it follows that the common neighborhood of $L'$ has size at least $\left(1-\frac{1}{8k^2}\right)n$. That is, there are at most $\frac{n}{8k^2}$ vertices not adjacent to all of $L'$. Call this set of ``exceptional vertices" $E$. That is,
\[
E: = \{v\in V(H_k) \setminus L': |N_1(v) \cap L'| \leq k-1\}.
\]

Let $R = V(H_k) \setminus (L' \cup E)$ be the remaining vertices. So we have that $V(H_k)$ is the disjoint union of $L'$, $R$, and $E$ with $|L'| = k$ and $|E| \leq \frac{n}{8k^2}$. In the next two lemmas we will show that $E = \emptyset$ and this will allow us to prove Theorems \ref{thm spectral turan even cycles} and \ref{thm spectral turan subsequent cycles}. Note that because $R$ has size larger than $2k+2$, adding a new vertex adjacent  only to the vertices in $L'$ cannot create a $C_{2k+2}$, otherwise there would have already been one. 

\begin{lemma}
\label{minimum eigenweight in the neighborhood of a vertex}
For any vertex $u \in V(H_k)$, the Perron weight in the neighborhood of $u$ satisfies  $\sum_{w \sim u}\mathbf{\mathrm{v}}_w \geq k - \frac{1}{16k^2}$. 
\end{lemma}

\begin{proof}
Assume to the contrary that there exists a vertex $u$ with $\sum_{w \sim u}\mathbf{\mathrm{v}}_w < k - \frac{1}{16k^2}$. Note that because $\sum_{w \sim u}\mathbf{\mathrm{v}}_w=\lambda \mathbf{\mathrm{v}}_u \geq \sqrt{kn}\mathbf{\mathrm{v}}_u$, we have that $u\not \in L'$. Now modify the neighborhood of $u$ by deleting all the edges adjacent to $u$ and joining $u$ to all the vertices of $L'$. Call the resultant graph, $H_k^*$. The neighborhood of $u$ in $H_k^*$ has Perron weight at least $k - \frac{1}{16k^2}$ by Lemma \ref{precise size of L} thus, $\lambda(H_k^*) > \lambda(H_k)$ by the Rayleigh quotient. Moreover, $H_k^*$ does not contain any $C_{2k+2}$, a contradiction.
\end{proof}

With this we may show that $E$ is empty.

\begin{lemma}\label{E empty set}
The set $E$ is empty and $H_k$ contains the complete bipartite graph $K_{k, n-k}$.
\end{lemma}

\begin{proof}
Assume to the contrary that $E \neq \emptyset$. 
Note that $e(R) \leq 1$ and every vertex in $E$ has at most one neighbor in $R$, else we can embed a $C_{2k+2}$ in $H_k$. %For any vertex $u$ in $R$ then, we have $k\leq d(u) \leq k+1$ and so
%\begin{equation}\label{perron weight in L for E vertices}
%\sqrt{kn}\mathbf{\mathrm{v}}_u \leq \lambda \mathbf{\mathrm{v}}_u = \sum_{w\sim u} \mathbf{\mathrm{v}}_w \textcolor{red}{\leq k+1},
%\end{equation}
%and hence $\mathbf{\mathrm{v}}_u \leq \sqrt{\frac{2k}{n}}$. 
Any vertex $r \in R$, satisfies $\mathbf{\mathrm{v}}_r < \eta$. 
Therefore, for any vertex $u\in E$ we have that 
\[
\sum_{u\sim w} \mathbf{\mathrm{v}}_w = \lambda \mathbf{\mathrm{v}}_u = \sum_{\substack{w\sim u\\ w\in L'\cup R}} \mathbf{\mathrm{v}}_w + \sum_{\substack{w\sim u\\ w\in E}} \mathbf{\mathrm{v}}_w.
\]
By Lemma \ref{minimum eigenweight in the neighborhood of a vertex} %and \eqref{perron weight in L for E vertices} 
and using that vertices in $E$ have at most $k-1$ neighbors in $L'$, we have
\begin{equation}\label{Perron weight in E lower bound}
    \sum_{\substack{w\sim u\\ w\in E}} \mathbf{\mathrm{v}}_w \geq 1 - \frac{1}{16k^2} - \eta,
\end{equation}
Since the Perron weight in $E$ is at least $1 - \frac{1}{16k^2} - \eta > \frac{3}{4}$, the Perron weight outside $E$ is at most $k-1+\eta$, and the total Perron weight is $\lambda \mathbf{\mathrm{v}}_u$, we have that 
\[
\sum_{\substack{w\sim u\\ w\in E}} \mathbf{\mathrm{v}}_w  \geq \frac{3}{4k} \lambda \mathbf{\mathrm{v}}_u
\]

Now, applying Lemma \ref{lower bound for spectral radius of nonnegative matrices}, with $M = A(H_k[E])$, and $\mathbf{\mathrm{y}} = \mathbf{\mathrm{v}}_{|_E}$ (the restriction of $\mathbf{\mathrm{v}}$ to the set $E$), we have that for any $u \in E$,

\begin{equation}
M \mathbf{\mathrm{y}}_u = \sum_{\substack{w \sim u \\ w \in E}} \mathbf{\mathrm{v}}_w \geq \frac{3}{4k} \lambda \mathbf{\mathrm{v}}_u = \frac{3}{4k} \lambda \mathbf{\mathrm{y}}_u.  
\end{equation}

Hence, by Lemma \ref{lower bound for spectral radius of nonnegative matrices}, $\lambda(M) \geq \frac{3}{4k}\lambda \geq \frac{3}{4}\sqrt{\frac{n}{k}}$. This contradicts Lemma \ref{bounds on lambda} because $\lambda(M) \leq \sqrt{2k(|E|-1)} < \sqrt{\frac{n}{4k}}$ as $E$ induces a $C_{2k+2}$-free graph. Thus, $E = \emptyset$. Therefore $R = V(H_k) \setminus L'$ and $H_k$ must contain a $K_{k, n-k}$.
\end{proof}

\section{Proofs of Theorems \ref{thm spectral turan even cycles} and \ref{thm spectral turan subsequent cycles}}
\label{section proofs for even cycle chapter}

With Lemma \ref{E empty set} in hand we may complete the proofs of Theorems \ref{thm spectral turan even cycles} and \ref{thm spectral turan subsequent cycles}. Lemma \ref{E empty set} gives us that $K_{k, n-k}$ is a subgraph of both $H_k$ and $H'_k$ and the results follow quickly from this.
\medskip

\noindent{\bfseries Proof of Theorem \ref{thm spectral turan even cycles}.} 
We have shown that $K_{k, n-k} \subset H_k$, where the part with $k$ vertices is $L'$ and the other part with $n-k$ vertices is $R$. Thus $\mathcal{E}(L', R) = \{lr | l \in L', r \in R\}$. Now we know that $e(R) \leq 1$ and $H_k[L']$ is isomorphic to some subgraph of $K_k$. Thus, $H_k$ is a subgraph of $S_{n,k}^+$ and by the monotonicity of the spectral radius over subgraphs, we have that $H_k \cong S_{n,k}^+$.
\qed

\medskip
\noindent{\bfseries Proof of Theorem \ref{thm spectral turan subsequent cycles}.} 
We have $\mathcal{E}(L', R) = \{lr | l \in L', r \in R\}$. In addition, $e(R) = 0$, otherwise we can embed a $C_{2k+1}$. Also, $H_k[L']$ is isomorphic to some subgraph of $K_k$. Thus, $H_k$ is a subgraph of $S_{n,k}$ and by the monotonicity of the spectral radius over subgraphs, we have that $H_k \cong S_{n,k}$.
\qed

	\bibliographystyle{plain}
	\bibliography{bib.bib}
\end{document}